\documentclass[reqno]{amsart} 

\usepackage{amssymb}
\usepackage{amsmath}
\usepackage{epic}
\usepackage{accents}
\usepackage{amsthm}
\usepackage{amssymb}
\usepackage{mathtools}
\usepackage{enumitem}
\usepackage{comment}

\usepackage{mathrsfs}

\newtheorem{theorem}{Theorem}[section]
\newtheorem{lemma}[theorem]{Lemma}

\newtheorem{proposition}[theorem]{Proposition}
\newtheorem{example}{Example}[section]
\newtheorem{corollary}[theorem]{Corollary}
\newtheorem{definition}{Definition}[section]

\def \<{\langle}
\def \>{\rangle}

\DeclareFontFamily{OMS}{smallo}{}
\DeclareFontShape{OMS}{smallo}{m}{n}{<->s*[.65]cmsy10}{}
\DeclareSymbolFont{smallo@m}{OMS}{smallo}{m}{n}
\DeclareMathSymbol{\smallo}{\mathord}{smallo@m}{79}

\copyrightinfo{}{}
\issueinfo{}
  {}
  {}
  {}

\PII{}
\def\@serieslogo{\@empty}

\title{The Igusa zeta function of restricted power series over $\mathbb{Q}_p$}
\author[]{Leonie Dausy}

\begin{document}

\begin{abstract}
In this article, we ask whether the Igusa zeta function of a restricted power series over $\mathbb{Q}_p$ can be determined solely from the terms of degree at most $D$. That is, we ask whether the truncated polynomial $f_D$, consisting of all terms of $f$ of degree $\leq D$, yields the same Igusa zeta function as $f$ for sufficiently large $D$. Our main results include a counterexample already in the one-variable case, but also a positive result under the condition that $f$ is sufficiently non-degenerate with respect to its Newton polyhedron $\Gamma(f)$.
\end{abstract}

\maketitle
\section{Introduction}
\noindent
Igusa introduced the local zeta function in the 1960s as a tool to study Diophantine problems. For a polynomial $f$ over  $\mathbb{Z}_p$, the local zeta function is defined by
   \begin{equation*}
        Z_f(s) = \int_{\mathbb{Z}_p} \vert f(x) \vert_p^s \vert dx \vert,
    \end{equation*}
where $\vert dx \vert$ denotes the Haar measure. This integral is directly related to counting the number of solutions of the congruences $f(x) \equiv 0 \mod p^m$, $m= 1, 2, 3, \dots$, see e.g. \cite{igusa1978lectures}, pp. 97 98. Using resolutions of singularities, Igusa proved that $Z_f(s)$ is a rational function of $p^{-s}$ (see also \cite{igusa1978lectures}). An entirely different proof was obtained ten years later by Denef \cite{denef1984rationality} using p-adic cell decomposition instead of resolutions of singularities. Despite these theoretical advances, explicitly computing $Z_f(s)$ for a given function $f$ remains a highly nontrivial task. Over the years, a number of results have been established for specific classes of polynomials. In this article, we extend the investigation to the setting of restricted power series over $\mathbb{Q}_p$\\ \\
{\bf Notations.} Fix a prime number $p$ and consider a tuple of $M$ variables $X = (X_1, \dots, X_M)$. For $i = (i_1, \dots, i_M)$ ranging over $\mathbb{N}^M$, we use the notation $\vert i \vert = i_1 + \dots + i_M$ and $X^{i}= X_1^{i_1}\dots X_M^{i_M}$. 
    \begin{definition}
        The ring $\mathbb{Q}_p \{X\}$ of restricted power series in $X$ over the $p$-adic field $\mathbb{Q}_p$ consists of the formal power series $\sum_{i}a_iX^{i} \in \mathbb{Q}_p[[X]]$ such that $\vert a_i \vert_p \rightarrow 0$ as $\vert i \vert \rightarrow \infty$. The subring $\mathbb{Z}_p \{X\}$ of $\mathbb{Q}_p \{X\}$ consists of the series $\sum_{i}a_ix^{i}$ all of whose coefficients $a_i$ are in $\mathbb{Z}_p$. 
    \end{definition}
\noindent
Let $f = \sum_{i}a_iX^{i}$ be in $\mathbb{Q}_p \{X\}$ and $x \in \mathbb{Z}_p$. Then the series $\sum_{i}a_ix^{i}$ converges to a limit in $\mathbb{Q}_p$
which we denote by $f(x)$. For any such analytic series, we can analogously define the Igusa Zeta functions as for polynomials.
   \begin{definition}
    Consider $f \in \mathbb{Z}_p\{X\}$ and $s \in \mathbb{C}$ with $\Re(s) > 0$. The Igusa zeta function of $f$ is defined as
        \begin{equation*}
            Z_f(s) = Z_p(f;s) \coloneqq \int_{\mathbb{Z}_p^n}\vert f(x) \vert_p^s \vert dx \vert
        \end{equation*}
    where $\vert dx \vert$ denotes the Haar measure.
    \end{definition}
    \begin{example}\label{integral_x^k}
A very standard example is the integral $\int_{\mathbb{Z}_p} \vert x^n \vert^s \vert dx \vert$. Denote $\{x \in \mathbb{Q}_p \, \vert \, \vert x \vert_p = p^{-l} \}$ by $A_l$ and observe that
    \begin{align*}
        \vert dx \vert (A_l) 
        & = \vert dx \vert \big( \{x \in \mathbb{Q}_p \, \vert \, \vert x \vert_p \leq p^{-l} \} \setminus \{x \in \mathbb{Q}_p \, \vert \, \vert x \vert_p \leq p^{-(l+1)} \} \big)  \\
        & = p^{-l} - p^{-l-1}  = \big( 1 - \frac{1}{p} \big)p^{-l}.
    \end{align*}
We can calculate $\int_{\mathbb{Z}_p} \vert x^n \vert^s \vert dx \vert$ as follows:
    \begin{align*}
        \int_{\mathbb{Z}_p} \vert x^n \vert^s \vert dx \vert & = \sum_{l \geq 0} \int_{A_l} \vert x^n \vert^s \vert dx \vert  = \sum_{l \geq 0} p^{-nsl} \int_{A_l} \vert dx \vert  = \sum_{l \geq 0} p^{-nsl}  \vert dx \vert(A_l)  \\
        & = \big( 1 - \frac{1}{p} \big)\sum_{l \geq 0} p^{-(ns+1)l}.
    \end{align*}
\end{example}
\noindent
{\bf Main question.} Consider a formal power series $f(X) = \sum_{i}a_iX^{i}$ in $\mathbb{Z}_p \{X\}$. Does there always exist $N \in \mathbb{N}$ such that for all $D \geq N$, the cut off $f_D(X) = \sum_{\vert i \vert \leq D}a_iX^{i}$ of $f$ at degree $D$ has the same Igusa Zeta function as $f$? \\ \\
{\bf Main result.}
A counterexample already exists in the one-variable case. Nonetheless, there is a partial positive result: the answer is yes when $f$ and $f'$ do not vanish simultaneously. This non-degeneracy condition extends to multiple variables, where the condition now becomes: non-degenerate over $\mathbb{F}_p$ with respect to the faces of the Newton polyhedron.

\section{Counterexample to the main question}\label{one_variable_counterexample}
\noindent
Consider the series
    \begin{equation*}
        f(x) = (x-1)^{2}\sum_{i = 0}^{\infty}p^{i}x^{i}.
    \end{equation*}
Since $\vert \sum_{i = 0}^{\infty}p^{i}x^{i} \vert_p = 1$ for all $x \in \mathbb{Z}_p$, we have $\int_{\mathbb{Z}_p} \vert f(x) \vert_p^s \vert dx \vert =\int_{\mathbb{Z}_p} \vert (x-1)^{2} \vert_p^s \vert dx \vert$.
Making a change of variables $x \rightarrow x+1$, we see $ \int_{\mathbb{Z}_p} \vert f(x) \vert_p^s \vert dx \vert =\int_{\mathbb{Z}_p}\vert x^{2} \vert_p^s \vert dx \vert$. By example (\ref{integral_x^k}), we get
 \begin{align}\label{integral_counterexample_onevariable}
        Z_f(s) = \int_{\mathbb{Z}_p} \vert f(x) \vert_p^s \vert dx \vert  = \Big( 1 -\frac{1}{p} \Big) \sum_{l \geq 0}p^{-(2s+1)l}.
    \end{align}
Next, an easy calculation shows
    \begin{align*}
        f_{2D}(x) = (x-1)^2\Big( \sum_{i = 0}^{2D-2}p^i x^i \Big) + p^{2D-1}x^{2D-1}(1 + (p-2)x).
    \end{align*}
Observe that if $x \not \equiv 1 \mod p$, then $(x-1)^2\Big( \sum_{i = 0}^{2D-2}p^i x^i \Big)$ has valuation zero and $p^{2D-1}x^{2D-1}(1 + (p-2)x)$ has valuation $\geq 2D-1$, so that $v_p(f_{2D}(x)) = 0$ or equivalently $\vert f_{2D}(x) \vert_p = 1$. Hence
    \begin{align*}
        \sum_{a \in \mathbb{F}_p \setminus \{1\}}\int_{a + p\mathbb{Z}_p}\vert f_{2D} \vert_p^s \vert dx \vert 
        & = \sum_{a \in \mathbb{F}_p \setminus \{1\}}\int_{a + p\mathbb{Z}_p}1 \vert dx \vert \\
        & = \sum_{a \in \mathbb{F}_p \setminus \{1\}}\vert dx \vert (a + p \mathbb{Z}_p) \\
        & = \sum_{a \in \mathbb{F}_p \setminus \{1\}}\frac{1}{p} =  1 - \frac{1}{p}.
    \end{align*}
Starting from the fact that $ \int_{\mathbb{Z}_p}\vert f_{2D} \vert_p^s \vert dx \vert = \sum_{a \in \mathbb{F}_p}\int_{a + p\mathbb{Z}_p}\vert f_{2D} \vert_p^s \vert dx \vert$, it only remains to calculate $\int_{1 + p\mathbb{Z}_p}\vert f_{2D} \vert_p^s \vert dx \vert$. Note that if $x \equiv 1 \mod p$, then $v_p(x) = 0$. Therefore 
    \begin{equation*}
        f_{2D}(x) = \underbrace{(x-1)^2\Big( \sum_{i = 0}^{2D-2}p^i x^i \Big)}_{v_p = 2v_p(x-1)} + \underbrace{p^{2D-1}x^{2D-1}(1 + (p-2)x)}_{v_p = 2D-1}
    \end{equation*}
Since $2v_p(x-1)$ is even and $2D-1$ is odd, one of those two terms will always dominate. More precisely, if we define
    \begin{equation*}
        B_l \coloneqq \{x \in \mathbb{Z}_p \, \vert \, v_p(x-1) = l\},
    \end{equation*}
it holds that
    \begin{equation*}
        \vert f_{2D}(x) \vert_p = 
            \begin{cases}
                p^{-2l} \quad \quad \quad \,\text{if} \, x \in B_l \,\, \text{for} \,\, 1 \leq l < D \\
                p^{-(2D-1)} \quad \,\text{if} \, x \in B_l \, \,\text{for} \,\, l \geq D.
            \end{cases}
    \end{equation*}
Because $1 + p \mathbb{Z}_p = \sqcup_{l=1}^{\infty}B_l$, we have
    \begin{align*}
        \int_{1 + p\mathbb{Z}_p} \vert f_{2D} \vert_p^s \vert dx \vert
        & = \sum_{l = 1}^{\infty} \int_{B_l} \vert f_{2D} \vert_p^s \vert dx \vert =  
         \underbrace{\sum_{l = 1}^{D-1} \int_{B_l} \vert f_{2D} \vert_p^s \vert dx \vert }_{(1)} + \underbrace{\sum_{l = D}^{\infty} \int_{B_l} \vert f_{2D} \vert_p^s \vert dx \vert}_{(2)}.
    \end{align*}
Then
\begin{alignat*}{4}
\text{(1)} 
& = \sum_{l = 1}^{D-1} p^{-2sl} \int_{B_l} \vert dx \vert 
& \quad\quad\quad
\text{(2)} 
& = \sum_{l = D}^{\infty} p^{(-2D-1)s} \int_{B_l} \vert dx \vert \\
& = \sum_{l = 1}^{D-1} p^{-2sl} \vert dx \vert(B_l) 
& \quad\quad\quad
& = p^{(-2D-1)s} \sum_{l = D}^{\infty} \int_{B_l} \vert dx \vert \\
& = \sum_{l = 1}^{D-1} p^{-2sl} \Big(1 - \frac{1}{p}\Big)p^{-l} 
& \quad\quad\quad
& = p^{(-2D-1)s} \int_{1 + p^D\mathbb{Z}_p} \vert dx \vert \\
& = \Big(1 - \frac{1}{p} \Big) \sum_{l = 1}^{D-1} p^{(-2s+1)l} 
& \quad\quad\quad
& = \frac{p^{-(2D+1)s}}{p^{D}}.
\end{alignat*}
All together, we get
    \begin{equation*}
        Z_{f_{2D}}(s) = \Big(1 - \frac{1}{p} \Big) \sum_{l = 0}^{D-1} p^{(-2s+1)l} + \frac{p^{-(2D+1)s}}{p^{D}}.
    \end{equation*}
For no value of $D$ is this equal to (\ref{integral_counterexample_onevariable}) and we are done. 

\section{The non-degenerate case in one variable}
\noindent
Let $f(x) = \sum_i a_i x^{i} \in \mathbb{Z}_p \{x\}$ be a restricted power series in one variable. We begin by zooming in and considering $p$-adic integrals of the form $\int_{\alpha + p^m \mathbb{Z}_p} \vert f(x) \vert_p^s \vert dx \vert$. For these integrals, we can get nice expressions if $m$ is sufficiently large.  We will distinguish three cases: 
    \begin{enumerate}
        \item $\alpha$ such that $f$ does not vanish on $\alpha + p^m \mathbb{Z}_p$
        \item $\alpha$ such that $f(\alpha) = 0$ but $f'(\alpha) \neq 0$
        \item $\alpha$ such that $f(\alpha) = f'(\alpha) = 0$. 
    \end{enumerate}
In the first two cases, we indeed have  $\int_{\alpha + p^m \mathbb{Z}_p} \vert f(x) \vert_p^s \vert dx \vert =  \int_{\alpha + p^m \mathbb{Z}_p} \vert f_D(x) \vert_p^s \vert dx \vert$ for all $D$ sufficiently large. Hence if $f$ and $f'$ have no common zeros, we will get a positive answer. If an $\alpha$ of the third type exists, this conclusion no longer holds. The counterexample discussed in the previous section already illustrates this point. Note that in that case, we have $f(1) = f'(1) = 0$. 
\subsection{Away from a zero}\label{away from zero}
Consider a neighborhood $\alpha + p^m \mathbb{Z}_p$ such that $f$ does not vanish on it. Since closed balls are compact in the $p$-adic topology, the continuous function $\vert f \vert_p: \mathbb{Z}_p \rightarrow \mathbb{R}$ will reach its minimum on the compact set $\alpha + p^m \mathbb{Z}_p$. This cannot be zero by the choice of the neighbourhood, so we denote it with $\epsilon > 0$. Choose $D_{\alpha}$ large enough such that $\vert a_i \vert_p < \epsilon$ for all $i \geq D_a$. Consequently, for all $D > D_{\alpha}$
    \begin{align*}
        \vert(f-f_D)(x)\vert_p = \vert \sum_{i = D+1}^{\infty}a_i x^{i}\vert_p < \epsilon.
    \end{align*}
Fix $x \in \alpha + p^m\mathbb{Z}_p$. Then $\vert f(x) \vert_p \geq \epsilon$. If $\vert f_D(x) \vert_p < \epsilon$, it would follow that $\epsilon \leq \vert f(x) \vert_p  \leq \max( \vert f_D(x) \vert_p, \vert (f-f_D)(x)\vert_p) < \epsilon$. This is a contradiction, so we may conclude $\vert f_D(x) \vert_p \geq \epsilon > \vert (f-f_D)(x)\vert_p$. Consequently
    \begin{align*}
        \vert f(x) \vert_p  = \max( \vert f_D(x) \vert_p, \vert (f-f_D)(x)\vert_p) = \vert f_D(x) \vert_p.
    \end{align*}
We have equality of the integrands, so their integrals will be equal as well:
    \begin{equation*}
        \int_{\alpha + p^m \mathbb{Z}_p}\vert f(x) \vert_p^s  \vert dx \vert = \int_{\alpha + p^m \mathbb{Z}_p}\vert f_D(x) \vert_p^s  \vert dx \vert.
    \end{equation*}
\subsection{In the neighborhood of a zero with nonzero derivative}
The goal is to compute integrals of the form $\int_{\alpha + p^m \mathbb{Z}_p} \vert f(x) \vert_p^s \vert dx \vert$. After a change of variables $x \rightarrow \alpha + p^mx$, we get
    \begin{equation*}
        \int_{\alpha + p^m\mathbb{Z}_p} \vert f(x) \vert_p^s \vert dx \vert = p^{-m} \int_{\mathbb{Z}_p} \vert f(\alpha + p^m x) \vert_p^s \vert dx \vert. 
    \end{equation*}
We can formally define the derivative of $f$ as $f'(x) = \sum_{i \geq n}ia_iX^{i-1}$. Because $\vert a_i \vert_p \rightarrow 0$ as $i \rightarrow \infty$, we definitely have $\vert i a_i \vert_p \rightarrow 0$ as $i \rightarrow \infty$ and thus $f'(x) \in \mathbb{Z}_p \{X\}$. Inductively we can define any derivative $f^{(k)} \in \mathbb{Z}_p \{X\}$. Then we have the Taylor series
    \begin{align*}
        f(\alpha + p^m x) 
        & = f(\alpha) + p^m f^{(1)}(\alpha) x + p^{2m} \frac{f^{(2)}(\alpha)}{2!}x^2 + p^{3m} \frac{f^{(3)}(\alpha)}{3!}x^3 + \dots
    \end{align*}
Let $n_0$ be the smallest natural number so that $f^{(n_0)}(\alpha) \neq 0$. If we choose $m$ large enough so that $v_p(\frac{f^{(n_0)}}{n_0!}(\alpha)) < m$, the coefficient of the term $p^{m n_0}\frac{f^{(n_0)}}{n_0!}(\alpha)x^{n_0}$ will have a uniquely maximal $p$-adic norm among all the coefficients of $f(\alpha + p^m x)$. Consequently, 
    \begin{align}\label{norm_f_neighbourhood}
        \vert f(\alpha + p^m x) \vert_p = \vert p^{m n_0}\frac{f^{(n_0)}}{n_0!}(\alpha) x^{n_0}\vert_p
    \end{align}
for all $x \in \mathbb{Z}_p$. 
 For readability, we denote $b = \frac{f^{(n_0)}}{n_0!}(\alpha)$. We can then calculate
    \begin{align*}
        \int_{\alpha + p^m\mathbb{Z}_p} \vert f(x) \vert_p^s \vert dx \vert 
        & = p^{-m} \int_{\mathbb{Z}_p} \vert f(\alpha + p^m x) \vert_p^s \vert dx \vert \\
        & = p^{-m} \int_{\mathbb{Z}_p} \vert p^{m n_0}b x^{n_0}\vert_p^s \vert dx \vert \\
        & =  p^{-m}p^{-mn_0 s}\vert b\vert_p^s \int_{\mathbb{Z}_p} \vert x^{n_0}\vert_p^s \vert dx \vert.
    \end{align*}
By Example \ref{integral_x^k}, we know that $ \int_{\mathbb{Z}_p} \vert x^{n_0}\vert_p^s \vert dx \vert = \big( 1 - \frac{1}{p} \big) \frac{1}{1-p^{-(n_0s+1)}}$ and thus
    \begin{align}\label{integral_f_neighbourhood}
        \int_{\alpha + p^m\mathbb{Z}_p} \vert f(x) \vert_p^s \vert dx \vert 
        = (p^{-m}  -p^{-m+1}) \frac{p^{-mn_0 s}\vert b\vert_p^s}{1-p^{-(n_0s+1)}}.
    \end{align}
Note that this only depends on $n_0$ and $\vert b\vert_p = \vert \frac{f^{(n_0)}(\alpha)}{n_0!} \vert_p$. \\ \\
For this case, we consider $n_0 = 1$. Denote $e = v_p(f'(\alpha))$. To use the discussion above, fix $m > e$. We want to use some version of Hensel's Lemma.
    \begin{lemma}{(Hensel)}
        Let $f \in \mathbb{Z}_p\{X\}$ be an analytic series in one variable $X$ over $\mathbb{Z}_p$. Let $a \in \mathbb{Z}_p$ such that $f(a) \equiv 0 \mod p$ and $f'(a) \not \equiv 0 \mod p$. Then there exists a unique $\xi \in \mathbb{Z}_p$ such that
            \begin{align*}
                f(\xi) = 0  \quad \text{and} \quad \xi \equiv a \mod p.
            \end{align*}
    \end{lemma}
\begin{proof}
    See for example \cite{koblitz2012p} pp16-17. Note that here the proof is for polynomials $f \in \mathbb{Z}_p[X]$, but the proof also works for $f \in \mathbb{Z}_p\{X\}$. 
\end{proof}
\noindent
We have the following variation:
    \begin{lemma}\label{generalised hensel}
        Let $a \in \mathbb{Z}_p$ such that $f(a) \equiv 0 \mod p^{2e+k}$ and $v_p(f'(a)) = e$. Then there exists a unique $\xi$ such that
            \begin{align*}
                f(\xi) = 0 \quad \text{and} \quad \xi \equiv a \mod p^{e+k}
            \end{align*}
    \end{lemma}
\begin{proof}
    Suppose $f$ and $a$ are such as in the Lemma. Consider $g(y) \in \mathbb{Q}_p\{y\}$ defined by
        \begin{equation*}
            g(y) = p^{-2e-(k-1)}f(a + p^{e+k-1}y).
        \end{equation*}
    Its Taylor expansion at $y= 0$ yields $h(y) \in \mathbb{Z}_p \{y\}$ such that
        \begin{equation*}
            g(y) = p^{-2e-(k-1)}\Big(f(a) + (p^{e+k-1}y)f'(a) + (p^{e+k-1}y)^2h(y)\Big).
        \end{equation*}
    Using $f(a) \equiv 0 \mod p^{2e+k}$ and $v_p(f'(a)) = e$, we see that actually $g(y) \in \mathbb{Z}_p\{y\}$. Moreover $g(0) \equiv 0 \mod p$ and $g'(0) = p^{-e}f(a) \not \equiv 0 \mod p$. 
    So we can apply Hensel's Lemma to find a unique $\eta \in \mathbb{Z}_p$ such that $\eta \equiv 0 \mod p$ and $g(\eta) = 0$. Then $\xi = a + p^{e+k-1}\eta$ satisfies $f(\xi) = 0$ and $\xi \equiv a \mod p^{e+k}$.
\end{proof}
\noindent
Now is the idea to choose $N$ large enough such that for all $D \geq N$ there exists $\xi_D \in \mathbb{Z}_p$ satisfying $f_D(\xi_D) = 0$, $v_p( f_D'(\xi_D)) = v_p(f'(\alpha)) = e$ and $\xi_D + p^m \mathbb{Z}_p = \alpha + p^{m} \mathbb{Z}_p$. 
 Then by (\ref{integral_f_neighbourhood}) we can conclude
        \begin{align*}
             \int_{\alpha + p^m\mathbb{Z}_p} \vert f_D(x) \vert_p^s \vert dx \vert 
            & = \int_{\xi_D + p^m\mathbb{Z}_p} \vert f_D(x) \vert_p^s \vert dx \vert \\
            &= (p^{-m}  -p^{-m+1}) \frac{p^{-ms}\vert f_D'(\xi_D)\vert_p^s}{1-p^{-(s+1)}} \\
            & = (p^{-m}  -p^{-m+1}) \frac{p^{-ms}\vert f(\alpha)\vert_p^s}{1-p^{-(s+1)}} \\
            &  = \int_{\alpha + p^m\mathbb{Z}_p} \vert f(x) \vert_p^s \vert dx \vert. 
        \end{align*}
Consider $D$ such that $v_p(a_i) > e+m$ for all $i \geq D$. Then $f_D \equiv f \mod p^{e+m}$ and also $f_D' \equiv f' \mod p^{e+1}$. Consequently, it must be that
    \begin{align*}
         & f_D(\alpha) \equiv f(\alpha) \mod p^{e+m} \equiv 0 \mod p^{e+m}\\
         & v_p(f_D'(\alpha)) = e.
    \end{align*}
Therefore, we can use Lemma \ref{generalised hensel} for $k = m-e$ to get $\xi_D \in \mathbb{Z}_p$ such that
    \begin{align*}
        f_D(\xi_D) = 0 \quad \text{and} \quad \xi \equiv \alpha \mod p^{m}.
    \end{align*}
Since $m > e$ and thus $\xi_D \equiv \alpha \mod p^{e+1}$ we have $f_D'(\xi_D) \equiv f_D(\alpha) \mod p^{e+1}$. We may conclude
    \begin{equation*}
        v_p(f_D'(\xi_D)) = e.
    \end{equation*}
This $\xi_D$ is exactly what we needed, so we are done. \\ \\
All together we get the following positive result.
    \begin{proposition}\label{result_one_variable_nondegenerate}
        Let $f \in \mathbb{Z}_p \{X\}$ such that $f$ and $f'$ have no common zeros. Then there exists $N$ such that for all $D \geq N$, the truncation $f_D$ has the same Igusa Zeta function as $f$.
    \end{proposition}
\begin{proof}
    From [\cite{cluckers2011fields}, Theorem 5.5.2] we may conclude $f$ only has finitely many roots in $\mathbb{Z}_p$. For each root $\alpha$, we can find $m_{\alpha}$ and $D_{\alpha}$ such that for all $m > m_{\alpha}$ and all $D \geq D_{\alpha}$
        \begin{equation*}
            \int_{\alpha + p^m \mathbb{Z}_p} \vert f(x) \vert_p^s \vert dx \vert= \int_{\alpha + p^m \mathbb{Z}_p} \vert f_D(x) \vert_p^s \vert dx \vert. 
        \end{equation*}
    Fix $m > \max{(m_{\alpha})}$. Denote
        \begin{align*}
            & V \coloneqq \{ a \in \frac{\mathbb{Z}}{p^m \mathbb{Z}_p} \, \vert \, a + p^m \mathbb{Z}_p = \alpha + p^m \mathbb{Z}_p \, \text{for some root $\alpha$}\} \\
            & S \coloneqq \frac{\mathbb{Z}}{p^m \mathbb{Z}_p}  \setminus V
        \end{align*}
    Since $V \cup S = \frac{\mathbb{Z}}{p^m \mathbb{Z}_p}$, we have the equality
        \begin{equation*}
            \int_{\mathbb{Z}_p} \vert f(x) \vert_p^s \vert dx \vert = \sum_{a \in V} \int_{a + p^m\mathbb{Z}_p} \vert f(x) \vert_p^s \vert dx \vert + \sum_{a \in S} \int_{a + p^m\mathbb{Z}_p} \vert f(x) \vert_p^s \vert dx \vert
        \end{equation*}
    For $a \in S$, note that $f$ does not vanish on $a + p^m \mathbb{Z}_p$. By the discussion above, there exists $D_a$ such that for all $D \geq D_a$
        \begin{equation*}
            \int_{a + p^m\mathbb{Z}_p} \vert f(x) \vert_p^s \vert dx \vert = \int_{a + p^m\mathbb{Z}_p} \vert f_D(x) \vert_p^s \vert dx \vert .       \end{equation*}
    So for all $D > \max{(D_{\alpha}, D_a)}$ we indeed get
        \begin{equation*}
             \int_{\mathbb{Z}_p} \vert f(x) \vert_p^s \vert dx \vert =  \int_{\mathbb{Z}_p} \vert f_D(x) \vert_p^s \vert dx \vert.
        \end{equation*}
    \end{proof}

\section{non-degenerate case in multiple variables}
\noindent
This section closely follows the paper \cite{denef2001newton} by Denef and Hoornaert. Sections \ref{Definitions}-\ref{Why being non-degenerated makes it easier to calculate the Igusa Zeta function} primarily recap material already presented there. We include selected proofs and results from that paper to help the reader in understanding the core ideas. In Section \ref{you can cut off}, we establish a result analogous to Proposition \ref{result_one_variable_nondegenerate}, but in the setting of $p$-adic analytic functions that are non-degenerate over $\mathbb{F}_p$ with respect to the faces of their Newton polyhedron.
\subsection{Definitions and Newton polyhedra}\label{Definitions}
An important class of p-adic analytic functions are the ones that are non-degenerated over $\mathbb{F}_p$ with respect to all the faces of their Newton polyhedron. In order to define this, we first need the following definitions:
    \begin{definition}\label{newton_polyhedron}
         Let $f(x) = f(x_1, \dots, x_n) = \sum_{\omega \in \mathbb{N}^n} a_{\omega}x^{\omega}$ be a non-zero analytic series over $\mathbb{Z}_p$ with $f(0) = 0$. Let $\mathbb{R}^{+} = \{x \in \mathbb{R} \, \vert \, x \geq 0\}$ and $\text{supp}(f) = \{\omega \in \mathbb{N}^n \, \vert \, a_{\omega} \neq 0\}$.The Newton polyhedron $\Gamma(f)$ of $f$ is defined as the convex hull in $(\mathbb{R}^{+})^n$ of the set
            \begin{equation*}
                \bigcup_{\omega \in \text{supp}(f)}\omega + (\mathbb{R}^{+})^n.
            \end{equation*} 
    \end{definition}

    \begin{definition}\label{face_poly}
        Let $p$ be a prime number. Let $f$ be as in Definition \ref{newton_polyhedron}. For every face $\tau$ of the Newton polyhedron $\Gamma(f)$ of $f$, we define
            \begin{equation*}
                f_{\tau}(x) = \sum_{\omega \in \tau}a_{\omega}x^{\omega}
            \end{equation*}
        and $\overline{f}_\tau(x)$ with coefficients in $\mathbb{F}_p$ by reducing each coefficient $a_{\omega}$ of $f_{\tau}$ modulo $p\mathbb{Z}_p$.
    \end{definition}
\noindent
Having these notions, we are ready to define what it means to be non-degenerate over $\mathbb{F}_p$ (or $\mathbb{Q}_p$) with respect to all the faces of the Newton polyhedron:

\begin{definition}
    Let $f$ and $f_{\tau}$ be as in definitions \ref{newton_polyhedron} and  \ref{face_poly}. We say $f$ is non-degenerated over $\mathbb{F}_p$ with respect to all the faces of its Newton polyhedron $\Gamma(f)$ if for every face $\tau$ of $\Gamma(f)$ the locus of the polynomial $\overline{f}_\tau(x)$ has no singularities in $(\mathbb{F}_p^{\times})^n$ or equivalently, the set of congruences 
        \begin{equation*}
            \begin{cases}
                f_{\tau}(x) \, \, \equiv 0 \mod p \\
                \frac{\partial f_{\tau}}{\partial x_i}(x) \equiv 0 \mod p, \quad  i = 1, \dots, n
            \end{cases}
        \end{equation*}
    has no solution in $(\mathbb{Z}_p^{\times})^n$.
    \end{definition}
\noindent
Next, we introduce the notion of the cone associated to a face $\tau$, because we will need this later on. 
    \begin{definition}
        Let $f$ be as in Definition \ref{newton_polyhedron}. For $a \in (\mathbb{R}^{+})^n$, we define
            \begin{equation*}
                m(a) = \inf_{y \in \Gamma(f)} \{a \cdot y \}.
            \end{equation*}
    \end{definition}
    \begin{definition}
        Let $f$ be as in Definition \ref{newton_polyhedron} and $a \in (\mathbb{R}^{+})^n$. We define the first meet locus of $a$ as
            \begin{equation*}
                F(a) = \{ z \in \Gamma(f) \, \vert \, a \cdot z = m(a)\}.
            \end{equation*}
    \end{definition}
    \noindent
    {\textbf{Fact.}}
        $F(a)$ is a face of $\Gamma(f)$. In particular, $F(0) = \Gamma(f)$ and $F(a)$ is a proper face of $\Gamma(f)$, if $a \neq 0$. Moreover $F(a)$ is a compact face iff $a \in (\mathbb{R}^{+}\setminus \{0\})^n$.
    \begin{definition}
        We define an equivalence relation on $(\mathbb{R}^{+})^n$ by
            \begin{equation*}
                a \sim a' \iff F(a) = F(a').
            \end{equation*}
        If $\tau$ is a face of $\Gamma(f)$, we define the cone associated to $\tau$ as
            \begin{equation*}
                \Delta_{\tau} = \{a \in (\mathbb{R}^{+})^n \, \vert \, F(a) = \tau \}.
            \end{equation*}
    \end{definition}
\noindent
Observe that the cones associated to the faces of $\Gamma(f)$ form a partition of $(\mathbb{R}^{+})^n$. So $(\mathbb{R}^{+})^n$ is equal to the disjoint union
    \begin{equation}\label{cones_associated_to_faces_partition}
        \{0\} \cup \bigcup_{\substack{\tau \, \text{proper face} \\ \text{of}\, \Gamma(f)}} \Delta_\tau.
    \end{equation}
Consequently, for every $k \in \mathbb{N}^n$, there is exactly one face $\tau$ such that $k \in \Delta_{\tau}$. This means that $F(k) = \tau$ or thus $ \tau = \{ z \in \Gamma(f) \, \vert \, a \cdot z = m(k)\}$.
Then also
    \begin{equation*}
        \tau \cap \text{supp}(f) = \{ \omega \in \text{supp}(f) \, \vert \,  k \cdot \omega = \inf_{y \in \Gamma(f)} k \cdot y \}
    \end{equation*} 

\subsection{Why being non-degenerated makes it easier to calculate the Igusa Zeta function}\label{Why being non-degenerated makes it easier to calculate the Igusa Zeta function}
A very useful formula is the following:
    \begin{proposition}\label{proposition}
        Let $p$ be a prime number and $f(x) = f(x_1, \dots, x_n) \in \mathbb{Z}_p\{x_1, \dots, x_n\}$. Denote by $\overline{f}$ the function over $\mathbb{F}_p$ obtained by reducing each coefficient modulo $p \mathbb{Z}_p$. Let $N$ be the number of elements in the set $\{a \in (\mathbb{F}_p^{\times})^n \, \vert \, \overline{f}(a) = 0\}$. Suppose that the set of congruences
            \begin{equation*}
                \begin{cases}
                    f(x)  \, \, \, \, \equiv 0 \mod p \\
                    \frac{\partial f}{\partial x_i}(x) \equiv 0 \mod p, \quad i =1, \dots, n
                \end{cases}
            \end{equation*}
        has no solution in $(\mathbb{Z}_p^{\times})^n$. Then for $s$ a complex variable with $\Re(s) > 0$, one has that
            \begin{equation*}
                \int_{(\mathbb{Z}_p^{\times})^n} \vert f(x) \vert_p^s \vert dx \vert  = p^{-n} \Big((p-1)^{n} - pN \frac{p^s-1}{p^{s+1}-1} \Big)
            \end{equation*}
        where $\vert dx \vert $ denotes the Haar measure on $\mathbb{Q}_p^n$ so normalized that $\mathbb{Z}_p^n$ has measure 1
    \end{proposition}

    \begin{proof}
        A proof of this is given in (\cite{denef2001newton}) in case $f$ is a polynomial. It is based on Hensel's Lemma, so since Hensel's Lemma also holds for analytic functions, this proof can be copied for $f$ an analytic series. 
    \end{proof}
\noindent
Note that if $f$ is non-degenerated over $\mathbb{F}_p$ with respect to all the faces of its Newton polyhedron, then for any analytic function $g$, it still holds that for every face $\tau$ of $\Gamma(f)$, the set of congruences
   \begin{equation*}
        \begin{cases}
            f_{\tau}(x) + pg(x) \, \,\equiv 0 \mod p \\
            \frac{\partial (f_{\tau} + pg)}{\partial x_i}(x) \quad \, \,  \equiv 0 \mod p, \quad i = 1, \dots, n
            \end{cases}
        \end{equation*}
has no solutions in $(\mathbb{Z}_p^{\times})^n$. Hence we can use the previous Proposition to calculate integrals of the form
    \begin{equation*}
       \int_{\mathbb{Z}_p^{\times}} \vert f_{\tau}(u) + p \tilde{f}_{\tau}(u) \vert_p^s \vert du \vert.
    \end{equation*}
The trick is then the following: in the previous section, we showed that for every $k = (k_1, \dots, k_n) \in \mathbb{N}^n$, there is a unique face $\tau$ of the Newton polyhedron such that
    \begin{equation*}
        \tau \cap \text{supp}(f) = \{ \omega \in \text{supp}(f) \, \vert \,  k \cdot \omega = \inf_{y \in \Gamma(f)} k \cdot y \},
    \end{equation*}
i.e. for $\omega \in \text{supp}(f)$, $k \cdot \omega$ reaches the infimum $m(k) = \inf_{y \in \Gamma(f)} k \cdot y$ if and only if $\omega \in \tau$. Now for $x \in \mathbb{Z}_p^n$ with $\text{ord} \,x=k$, we may write $x = p^k u$ for some unit $u \in (\mathbb{Z}_p^{\times})^n$. Denoting $\sigma(k) = k_1+\dots + k_n$, we have
    \begin{align*}
        &\vert dx \vert = p^{-\sigma(k)} \vert du \vert \\
        &x^{\omega} = p^{k \cdot \omega }u^{\omega}.
    \end{align*}
If $\omega \in \tau$, we get that $v_p(x^{\omega}) = k \cdot \omega = m(k)$ and hence $a_{\omega}x^{\omega}$ is certainly divisible by $p^{m(k)}$.
If $\omega \not \in \tau$, we get that $v_p(x^{\omega}) = k \cdot \omega > m(k)$ and hence $a_{\omega}x^{\omega}$ is certainly divisible by $p^{m(k) + 1}$. Consequently
    \begin{align*}
        f(x) 
        & = \sum_{\omega \in \text{supp}(f)}a_{\omega}x^{\omega} \\
        & = \sum_{\omega \in \text{supp}(f) \cap \tau}a_{\omega}x^{\omega} + \sum_{\omega \in \text{supp}(f) \setminus \tau}a_{\omega}x^{\omega} \\
        & = p^{m(k)}\Big(\sum_{\omega \in \text{supp}(f) \cap \tau}a_{\omega}u^{\omega}\Big) + p^{m(k)+1}\Big(\sum_{\omega \in \text{supp}(f) \setminus \tau}a_{\omega}p^{k \cdot \omega - m(k)-1}u^{\omega} \Big)\\
        & = p^{m(k)}(f_{\tau}(u) + p \tilde{f_\tau}(u)).
    \end{align*}
So we find that
    \begin{equation*}
        \int_{\substack{x \in (\mathbb{Z}_p)^n \\ v_p(x) = k}} \vert f(x) \vert_p^s \vert dx \vert = p^{-\sigma(k)-m(k) s}\int_{(\mathbb{Z}_p^{\times})^n} \vert f_{\tau}(u) + p \tilde{f_\tau}(u) \vert_p^s \vert du \vert
    \end{equation*}
and the latter is easily computed by Proposition \ref{proposition}.
Since
    \begin{equation*}
          \int_{x \in \mathbb{Z}_p^n } \vert f(x) \vert_p^s \vert dx \vert = \sum_{k \in \mathbb{N}^n} \int_{\substack{x \in \mathbb{Z}_p^n \\ v_p(x) = k}} \vert f(x) \vert_p^s \vert dx \vert,
    \end{equation*}
this should convince you that the Igusa Zeta function of analytic functions that are non-degenerated with respect to their Newton polynomial are easily compute using this method. More specifically, we have the following:
Remember that the cones associated to the faces of $\Gamma(f)$ form a partition of $(\mathbb{R}^{+})^n$, see (\ref{cones_associated_to_faces_partition}). Consequently $\mathbb{N}^n$ is the disjoint union , 
        \begin{equation*}
            \mathbb{N}^n =  \{0\} \cup \bigcup_{\substack{\tau \, \text{proper face} \\ \text{of}\, \Gamma(f)}} \mathbb{N}^n \cap \Delta_\tau.
        \end{equation*}
It follows that 
    \begin{align*}
        Z_f(s) 
        & =  \sum_{k \in \mathbb{N}^n} \int_{\substack{x \in \mathbb{Z}_p^n \\ v_p(x) = k}} \vert f(x) \vert_p^s \vert dx \vert \\
        & = \int_{(\mathbb{Z}_p^{\times})^n} \vert f(x) \vert_p^s \vert dx \vert + \sum_{\substack{\tau \, \text{proper} \\ \text{face of}\, \Gamma(f)}}\sum_{k \in \mathbb{N}^n \cap \Delta_{\tau}} \int_{\substack{x \in \mathbb{Z}_p^n \\ v_p(x) = k}} \vert f(x) \vert_p^s \vert dx \vert \\
        & =  \int_{(\mathbb{Z}_p^{\times})^n} \vert f(x) \vert_p^s \vert dx \vert + \sum_{\substack{\tau \, \text{proper} \\ \text{face of}\, \Gamma(f)}}\sum_{k \in \mathbb{N}^n \cap \Delta_{\tau}}p^{-\sigma(k)-m(k) s}\int_{(\mathbb{Z}_p^{\times})^n} \vert f_{\tau}(u) + p \tilde{f_\tau}(u) \vert_p^s \vert du \vert.
    \end{align*}
    \begin{theorem}\label{newton_formula_igusa}
        Let $p$ be a prime number. Let $f$ be as in Definition \ref{newton_polyhedron}. Suppose that $f$ is non-degenerated over the finite field $\mathbb{F}_p$ with respect to all the faces of its Newton polyhedron $\Gamma(f)$. Denote for every face $\tau$ of $\Gamma(f)$ 
            \begin{equation*}
                N_{\tau} = \# \{a \in (\mathbb{F}_p^{\times})^n \, \vert \, \overline{f}_{\tau}(a) = 0\}.
            \end{equation*}
        Let $s$ be a complex variable with $\Re(s) > 0$. Then
            \begin{align*}
                Z_f(s) = L_{\Gamma(f)} + \sum_{ \tau \, \text{proper face of} \, \Gamma(f)}L_{\tau}S_{\Delta \tau}
            \end{align*}
        with
            \begin{align*}
                L_{\tau} &= p^{-n} \Big((p-1)^{n} - pN_{\tau} \frac{p^s-1}{p^{s+1}-1} \Big) \\
                S_{\Delta \tau}& = \sum_{k \in \mathbb{N}^n \cap \Delta \tau}p^{-\sigma(k)-m(k)s}.
            \end{align*}
    \end{theorem}

\subsection{You can cut off}\label{you can cut off}
Note that the formula from Theorem (\ref{newton_formula_igusa}) only depends on $N_{\tau}$, which in turn is determined by $f \mod p$, and the Newton polyhedron $\Gamma(f)$. Therefore we will want to have $f_D \equiv f \mod D$ and $\Gamma(f_D) = \Gamma(f)$. Let us now focus on the latter condition. \\ \\
\noindent Let $f(x_1,\dots, x_n) = \sum_{\omega \in \mathbb{N}^n} a_{\omega}x^{\omega}$ be a non-zero analytic function over $\mathbb{Z}_p$ with $f(0) = 0$. Note that supp$(f)$ is not necessarily finite. However, we do have the following:
    \begin{lemma}
        There exist $\omega_1, \dots, \omega_k \in \text{supp}(f)$ such that 
            \begin{equation*}
                \bigcup_{\omega \in \text{supp}(f)} \omega + (\mathbb{R}^{+})^n = \bigcup_{i = 1}^k \omega_i + (\mathbb{R}^{+})^n
            \end{equation*}
    \end{lemma}
\begin{proof}
    We prove this with induction on $n$. The case $n= 1$ is trivial, so consider $n > 1$. It suffices to show that $\text{supp}(f) \subset \bigcup_{i = 1}^k \omega_i + (\mathbb{R}^{+})^n$. Pick any $\omega_1 = (\alpha_1, \dots, \alpha_n) \in \text{supp}(f)$. Observe that $\omega_1 + (\mathbb{R}^{+})^n$ already contains every $\omega = (w_1, \dots, w_n) \in \text{supp}(f)$ satisfying $w_i \geq \alpha_i$ for all $i$. Now we also want to add the $\omega = (w_1, \dots, w_n) \in \text{supp}(f)$ for which $w_i < \alpha_i$ for some index $i$. Let us start by adding all $\omega \in \text{supp}(f)$ of the form $(a, w_2, \dots, w_n)$ for some $a < \alpha_1$. So fix $a \in \{0, 1, \dots, \alpha_1-1\}$ and define
        \begin{equation*}
            A = \{ (w_2, \dots, w_n) \, \vert \, (a, w_2, \dots, w_n) \in \text{supp}(f)\}.
        \end{equation*}
    Note that this is equal to $\text{supp}(g)$ where $g$ is defined as follows: collecting all the terms in $f$ whose $x_1$-exponent is equal to $a$, we can write is as $x_1^{a}g(x_2, \dots, x_n)$. 
    By the induction hypothesis, there are $\eta_1, \dots, \eta_l \in \text{supp}(g)$ such that 
        \begin{equation*}
            \bigcup_{\eta \in \text{supp}(g)} \eta + (\mathbb{R}^{+})^{n-1} = \bigcup_{i = 1}^l \eta_i + (\mathbb{R}^{+})^{n-1}
        \end{equation*}
    It is not hard to see that then
        \begin{equation*}
            \bigcup_{i = 1}^l (a, \eta_i) + (\mathbb{R}^{+})^{n}
        \end{equation*}
    contains all $\omega \in \text{supp}(f)$ of the form $(a, w_2, \dots, w_n)$. We can apply this procedure for any $a \in \{0, 1, \dots, \alpha_1-1\}$ and subsequently repeat it finitely many times for $w_2 \in \{0, 1, \dots, \alpha_2-1\}, \dots, w_n \in \{0, 1, \dots, \alpha_n-1\}$ as well. This gives us $\{\omega_1, \dots, \omega_k\}$ such that
        \begin{equation*}
             \bigcup_{i = 1}^k \omega_k + (\mathbb{R}^{+})^{n}
        \end{equation*}
    contains $\text{supp}(f)$. This completes the argument. 
\end{proof}

\begin{corollary}
    Let $f(x) \in \mathbb{Z}_p\{X\}$ be a non-zero analytic function over $\mathbb{Z}_p$ with $f(0) = 0$. There exists $N_1$ such that the Newton polyhedron of $f$ is equal to the Newton polyhedron of the polynomial $f_{D}$ for all $D \geq N_1$.
\end{corollary}
\begin{proof}
    Pick $\omega_1, \dots, \omega_k$ as in the previous Lemma. Denote $\sigma(\omega) = \sigma(w_1, \dots, w_n) = w_1 + \dots + w_n$ and choose $N_1 \geq \max (\sigma(\omega_i))$. Then clearly $\{\omega_1, \dots, \omega_k\} \subset \text{supp}(f_D)$ for all $D \geq D_1$. Consequently 
    \begin{equation*}
        \bigcup_{\omega \in \text{supp}(f)} \omega + (\mathbb{R}^{+})^n = \bigcup_{i = 1}^k \omega_i + (\mathbb{R}^{+})^n \subset \bigcup_{\omega \in \text{supp}(f_D)} \omega + (\mathbb{R}^{+})^n \subset \bigcup_{\omega \in \text{supp}(f)} \omega + (\mathbb{R}^{+})^n
    \end{equation*}
    so that we may conclude $\bigcup_{\omega \in \text{supp}(f_D)} \omega + (\mathbb{R}^{+})^n = \bigcup_{\omega \in \text{supp}(f)} \omega + (\mathbb{R}^{+})^n$. Since $\Gamma(f_D)$ and $\Gamma(f)$ are the convex hull of this, we get that indeed $\Gamma(f_D)=\Gamma(f)$ for all $D \geq N_1$. 
    \end{proof}
\begin{proposition}
    Let $f \in \mathbb{Z}_p\{X\}$ be non-degenerated over $\mathbb{F}_p$ with respect to its Newton polyhedron. Denote $f_D$ for the polynomial obtained by only considering the terms of $f$ of degree $\leq D$. Then there exists $N$ such that for all $D \geq N$, the truncation $f_D$ has the same Igusa Zeta function as $f$.
\end{proposition}
\begin{proof}
Let $N_1$ be such that $\Gamma(f_D) = \Gamma(f)$ for all $D \geq N_1$. Let $N_2$ be such that $f_D \equiv f \mod p$ for all $D \geq N_2$. Then for all $D \geq \max (N_1, N_2)$, the statement follows from the formula in Theorem \ref{newton_formula_igusa}. 
\end{proof}
\noindent
\bibliographystyle{amsplain}
\bibliography{ref}

\end{document}